\documentclass[11pt]{article}
\usepackage{amsmath,amssymb}

\newtheorem{propo}{{\bf Proposition}}[section]
\newtheorem{coro}[propo]{{\bf Corollary}}
\newtheorem{lemma}[propo]{{\bf Lemma}} \newtheorem{theor}[propo]{{\bf
Theorem}} \newtheorem{ex}{{\sc Example}}[section]

\newenvironment{proof}{{\bf Proof.}}{$\Box$}
\def\R{{\mathbb R}}
\def\C{{\mathbb C}}

\begin{document}
\vspace*{1.0in}

\begin{center} ON ABELIAN SUBALGEBRAS AND IDEALS OF MAXIMAL DIMENSION IN SUPERSOLVABLE LIE ALGEBRAS

\end{center}
\bigskip

\centerline {Manuel Ceballos \footnote[1]{Supported by MTM2010-19336 and FEDER}} \centerline {Departmento de Geometria y Topologia,
Universidad de Sevilla} \centerline {Apartado 1160, 41080, Seville, Spain}
\centerline {and}

\centerline {David A. Towers} \centerline {Department of
Mathematics, Lancaster University} \centerline {Lancaster LA1 4YF,
England}
\bigskip

\begin{abstract}
In this paper, the main objective is to compare the abelian subalgebras and ideals of maximal dimension for finite-dimensional supersolvable Lie algebras. We characterise the maximal abelian subalgebras of solvable Lie algebras and study solvable Lie algebras containing an abelian subalgebra of codimension $2$. Finally, we prove that nilpotent Lie algebras with an abelian subalgebra of codimension $3$ contain an abelian ideal with the same dimension, provided that the characteristic of the underlying field is not two. Throughout the paper, we also give several examples to clarify some results.
\end{abstract}

\noindent {\it Mathematics Subject Classification 2010:} 17B05, 17B20, 17B30, 17B50. \\
\noindent {\it Key Words and Phrases:} Lie algebras, abelian subalgebra, abelian ideal, solvable, supersolvable, nilpotent.

\section{Introduction}

Nowadays, there exists an extensive body of research of Lie Theory due to its own importance from a theoretical point of view and also due to its applications to other fields
like Engineering, Physics and Applied Mathematics. However, some aspects of Lie
algebras remain unknown. Indeed, the classification of nilpotent and solvable Lie algebras is still an open problem, although the classification of certain other types of Lie algebras (like
semi-simple and simple ones) were already obtained in 1890, at least over the complex field. In
order to make progress on these and other problems, the need for studying
different properties of Lie algebras arises. For example, conditions on the lattice
of subalgebras of a Lie algebra often lead to information about the Lie algebra itself.
Studying abelian subalgebras and ideals of a finite-dimensional Lie algebra
constitutes the main goal of this paper.

Throughout $L$ will denote a finite-dimensional Lie algebra over a field $F$. The assumptions on $F$ will be specified in each result. Algebra direct sums will be denoted by $\oplus$, whereas vector space direct
sums will be denoted by $\dot{+}$. We consider the following invariants of $L$:
$$\alpha(L) = \max \{\, \dim (A) \, | \, A \,\, {\rm is \,\, an \,\, abelian \,\, subalgebra \,\,
of \,\, } L\},$$ $$\beta(L)  = \max \{\,
\dim (B) \, | \, B \,\, {\rm is \,\, an \,\,
abelian \,\, ideal \,\,of \,\,}L\}.$$
Both invariants are important for many reasons. For example, they are
very useful for the study of Lie algebra contractions
and degenerations. There is a large literature, in particular for
low-dimensional Lie algebras, see
\cite{GRH,BU10,NPO,SEE,GOR}, and the references given therein.

The first author dealing with the invariant $\alpha(\mathfrak{g})$
 was Schur \cite{Sch}, who studied in $1905$ the
abelian subalgebras of maximal dimension contained in the Lie
algebra of $n \times n$ square matrices. Schur proved that {\it
the maximum number of linearly independent commuting $n \times n$
matrices over an algebraically closed field is
$\left[\frac{n^2}{4}\right]+1$}, which is the maximal dimension of abelian
ideals of Borel subalgebras in the general linear Lie algebra $\mathfrak{gl}(n)$ ( where $\left[x \right]$ denotes
the integer part of a real number $x$).
Let us note that this result was
obtained only over an algebraically closed field such as the
complex number field. Almost forty years later, in 1944, Jacobson
\cite{jac} gave a simpler proof of Schur's results, extending
them from algebraically closed fields to arbitrary fields. This
fact allowed several authors to gain insight into the
abelian subalgebras of maximal dimension of many different types
of Lie algebras.

More specifically, for semisimple Lie algebras $\mathfrak{s}$ the invariant
$\alpha(\mathfrak{s})$ has been completely determined by Malcev
\cite{MAL}. Since there are no abelian ideals in $\mathfrak{s}$,
we have $\beta(\mathfrak{s})=0$.
The value of $\alpha$ for simple Lie algebras
 is reproduced in table $1$.  In this paper, we will study
several properties of these invariants and compare them for supersolvable, solvable and nilpotent Lie algebras.

\begin{table}[htp]
\caption{The invariant $\alpha$ for simple Lie algebras}\label{alphaparasimples}
\begin{center}
\begin{tabular}{|c|c|c|}
\hline $\mathfrak{s}$ & $\dim (\mathfrak{s})$ & $\alpha(\mathfrak{s})$ \\
\hline $A_n,\,n\ge 1$ & $n(n+2)$ & $\lfloor (\frac{n+1}{2})^2 \rfloor$ \\
\hline $B_3$ & $21$  & $5$  \\
\hline $B_n,\, n\ge 4$  & $n(2n+1)$ & $\frac{n(n-1)}{2}+1$ \\
\hline $C_n,\,n\ge 2$ & $n(2n+1)$ & $\frac{n(n+1)}{2}$ \\
\hline $D_n,\,n\ge 4$ & $n(2n-1)$ & $\frac{n(n-1)}{2}$ \\
\hline $G_2$ & $14$ & $3$ \\
\hline $F_4$ & $52$ & $9$ \\
\hline $E_6$ & $78$ & $16$ \\
\hline $E_7$ & $133$ & $27$ \\
\hline $E_8$ & $248$ & $36$ \\
\hline
\end{tabular}\end{center}
\end{table}

We shall call $L$ {\it supersolvable} if there is a chain $0=L_0 \subset L_1 \subset \ldots \subset L_{n-1} \subset L_n=L$, where $L_i$ is an $i$-dimensional ideal of $L$. The ideals $L^{(k)}$ of the {\it derived series} are defined by $L^{(0)} = L, L^{(k+1)} = [L^{(k)},L^{(k)}]$ for $k \geq 0$; we also write $L^2$ for $L^{(1)}$ and $L^3$ for $[L^2,L]$. It is well known that every supersolvable Lie algebra is also solvable. Moreover, these classes coincide over an algebraically closed field of characteristic zero (Lie's theorem). There are, however, examples of solvable Lie algebras over algebraically closed field of non-zero characteristic which are not supersovable (see for instance \cite[page 53]{jacob} or \cite{BAR}). The {\it Frattini ideal} of $L$, $\phi(L)$, is the largest ideal of $L$ contained in all maximal subalgebras of $L$. We will denote the {\it centre} of $L$ by $Z(L)=\{x \in L: [x,y]=0, \, \forall \, y \in L\}$ and the {\it centralizer} of a subalgebra $A$ of $L$ by $C_{L}(A)=\{x \in L: [x,A]=0\}$. Given a subalgebra $A$ of $L$, the {\it core} of $A$, denoted by $A_L$, is the largest ideal of $L$ contained in $A$. The {\it abelian socle} of $L$, Asoc$L$, is the sum of the minimal abelian ideals of $L$.

The structure of this paper is as follows. In section $2$ we give some bounds for the invariants $\alpha$ and $\beta$. In section $3$, we consider the classes of supersolvable, solvable and nilpotent Lie algebras $L$ with $\alpha(L) = n-1$ or $n-2$. In particular, we characterise $n$-dimensional solvable Lie algebras $L$ for which $\alpha(L)= n-2$ and prove that every supersolvable Lie algebra, $L$, of dimension $n$ with $\alpha(L)=n-2$ also satisfies $\beta(L)=n-2$. In the final section we show the $\alpha$ and $\beta$ invariants also coincide for nilpotent Lie algebras $L$ with $\alpha(L) = n-3$, provided that $F$ has characteristic different from two. We also give an example to show that the restriction on $F$ is necessary.

\section{Some bounds on $\alpha(L)$ and $\beta(L)$}
We shall call a $L$ {\it metabelian} if $L^2$ is abelian. First we have a bound on $\beta(L)$ for certain metabelian Lie algebras.

\begin{propo}\label{p:metab} Let $L$ be a metabelian Lie algebra of dimension $n$, and suppose that $\dim L^2 = k$. Then $\dim (L/C_L(L^2)) \leq [k^2/4]+1$. If, further, $L$ splits over $L^2$ then $\beta(L) \geq n - [k^2/4]-1$.
\end{propo}
\begin{proof} Let ad $\colon L \rightarrow$ Der $L^2$ be defined by ad\,$x(y) = [y,x]$ for all $y \in L^2$. Then ad is a homomorphism with kernel $C_L(L^2)$. It follows that $L/C_L(L^2) \cong D$ where $D$ is an abelian subalgebra of Der $L^2 \cong gl(n,F)$. It follows from Schur's Theorem on commuting matrices (see \cite{jac}) that $\dim (L/C_L(L^2) \leq [k^2/4]+1$.
\par
Now suppose that $L = L^2 \oplus B$ where $B$ is an abelian subalgebra of $L$. Then $C_L(L^2) = L^2 \oplus B \cap C_L(L^2)$ which is an abelian ideal of $L$.
\end{proof}
\bigskip

We call $L$ {\it completely solvable} if $L^2$ is nilpotent. Over a field of characteristic zero, every solvable Lie algebra is completely solvable. Next we note that if $L$ is completely solvable, has an abelian nilradical (so is metabelian) and the underlying field is perfect then $\alpha(L)$ and $\beta(L)$ are easily identified. If $\bar{F}$ is the algebraic closure of $F$ we put $\bar{S} = S \otimes_F \bar{F}$ for every subalgebra $S$ of $L$.

\begin{lemma}\label{l:ext} $\alpha(\bar{L}) \geq \alpha(L)$, $\beta(\bar{L}) \geq \beta(L)$.
\end{lemma}

\begin{lemma}\label{l:nilrad}
Let $L$ be any solvable Lie algebra with nilradical $N$. Then $C_L(N) \subseteq N$
\end{lemma}
\begin{proof} Suppose that $C_L(N) \not \subseteq N$. Then there is a non-trivial abelian ideal $A/(N \cap C_L(N)$ of $L/(N \cap C_L(N)$ inside $C_L(N)/(N \cap C_L(N)$. But now $A^3 \subseteq [A, N] = 0$, so $A$ is a nilpotent ideal of $L$. It follows that $A \subseteq N \cap C_L(N)$, a contradiction.
\end{proof}

\begin{theor}\label{t:abnil} If $F$ is a perfect field and $L$ is a completely solvable Lie algebra with abelian nilradical $N$ then $\alpha(L) = \beta(L) = \dim N$.
\end{theor}
\begin{proof} It is clear that $N$ is the unique maximal abelian ideal of $L$. Let $A$ be an abelian subalgebra of $L$ of maximal dimension. If $N \subseteq A$, then $A \subseteq C_L(N) = N$, by Lemma \ref{l:nilrad}, so $N = A$ and the result is clear, so suppose that $N \not \subseteq A$ and put $U = N + A$.
\par

Consider first the case where $F$ is algebraically closed and $\phi(L) = 0$.   Pick any $a \in A$ and put $C = N + Fa$.  Then $\phi(C) = 0$, by \cite[Theorem 2.5]{tow-var},
so $N \subseteq N(C)= {\rm Asoc} (C)$ by \cite[ Theorem 7.4]{frat}, and $N$ is completely reducible as an $Fa$-module.
Write $N = \oplus_{i=1}^{k} N_{i}$, where $N_{i}$ is an irreducible $Fa$-module for $1 \leq i \leq k$. Then the minimal polynomial of the restriction of ad $a$ to $N_{i}$ is irreducible
for each $i$, and so $\{({\rm ad} a)|_{N}: a \in A \}$ is a set of commuting diagonalizable operators. Thus $N =$ Asoc$L = Fn_1 + \ldots Fn_r$, where $Fn_i$ is a minimal ideal of $L$ for $1 \leq i \leq r$. If $n_i \in A$, then $C_U(n_i) = U$; if $n_i \notin A$ then $\dim C_U(n_i) \geq \dim U - 1$, since $U/C_U(n_i) \cong D$, where $D$ is a subalgebra of Der $Fn_i$. But
\[
N = C_L(N) \supseteq C_U(N) = \bigcap_{i=1}^r C_U(Fn_i),
\]
so
\[
\dim N \geq \dim U - (r - \dim (N \cap A)) = \dim A,
\]
and the result holds in this case.
\par

So suppose now that $\phi(L)$ is not necessarily trivial. Then $N/\phi(L)$ is the nilradical of $L/\phi(L)$, by \cite[Theorem 6.1]{frat}, and so $L/\phi(L)$ has abelian nilradical. Also $(A + \phi(L))/\phi(L)$ is an abelian subalgebra of $L/\phi(L)$. It follows from the above that
\[
\dim \left( \frac{A + \phi(L)}{\phi(L)} \right) \leq \dim \left( \frac{N}{\phi(L)} \right), \hspace{.2cm} \hbox{whence } \dim A \leq \dim N.
\]
Finally consider the case where $F$ is not necessarily algebraically closed. Since $F$ is perfect, $\overline{N(L)} = N(\bar{L})$, by \cite[page 42]{baki}. Hence
\[ \beta(L) \leq \alpha(L) \leq \alpha(\bar{L}) = \beta(\bar{L}) = \dim N(\bar{L})= \dim \overline{N(L)} = \dim N(L) = \beta(L)
\]
by Lemma \ref{l:ext} and the above.
\end{proof}
\bigskip

We obtain bounds for supersolvable Lie algebras by following a development similar to \cite[Lemma 2]{stew}.

\begin{lemma}\label{l:sup} Let $L$ be a supersolvable Lie algebra and let $A$ be a maximal abelian ideal of $L$. Then $C_L(A) = A$.
\end{lemma}
\begin{proof} We have that $C_L(A)$ is an ideal of $L$. Suppose that $C_L(A) \neq A$. Let $B/A$ be a minimal ideal of $L/A$ with $B \subset C_L(A)$. Then, for some $b \in B$, $B = A + Fb$, which is an abelian ideal of $L$, contradicting the maximality of $A$. The result follows.
\end{proof}

\begin{propo}\label{p:supsolv} Let $L$ be a supersolvable Lie algebra, $A$ any maximal abelian ideal of $L$. Suppose $\dim A = k$. Then $L/A$ is isomorphic to a Lie algebra of $k \times k$ lower triangular matrices.
\end{propo}
\begin{proof} Let ad $\colon L \rightarrow$ Der $A$ be defined by ad\,$x(y) = [y,x]$. Then ad is a homomorphism with kernel $C_L(A) = A$, by Lemma \ref{l:sup}. Since $L$ is supersolvable there is a flag of ideals $0 = A_0 \subset A_1 \subset \ldots \subset A_k = A$ of $L$. Choose a basis $e_1, \ldots , e_k$ for $A$ with $e_i \in A_i$. With respect to this basis the action of $L$ on $A$ is represented by $k \times k$ lower triangular matrices, since $[A_i, L] \subseteq A_i$ for each $0 \leq i \leq k$.
\end{proof}

\begin{coro}\label{c:supsolv} Let $L$ be a supersolvable Lie algebra with a maximal abelian ideal $A$ of dimension $k$. Then $\dim L \leq \frac{k(k+3)}{2}$ and $L$ has derived length at most $k + 1$.
\end{coro}
\begin{proof} The Lie algebra of $k \times k$ lower triangular matrices has dimension $\frac{k(k+1)}{2}$ and derived length $k$.
\end{proof}

\begin{coro}\label{c:supsolvbound} Let $L$ be a supersolvable Lie algebra of dimension $n$. Then
\[  \beta(L) \geq \left[ \frac{\sqrt{8n+9}-3}{2} \right].
\]
\end{coro}

\begin{coro}\label{solvbound} Let $L$ be a non-abelian solvable Lie algebra of dimension $n$ over an algebraically closed field of characteristic zero. Then
\[  \left[ \frac{\sqrt{8n+9}-3}{2} \right] \leq \alpha(L) \leq n-1.
\]
\end{coro}
\begin{proof} Simply use Corollary \ref{c:supsolvbound} and \cite[Proposition 2.5]{bc}.
\end{proof}

\section{Supersolvable Lie algebras with $\alpha(L) = n-1$ or $n-2$}
\begin{propo}\label{p:codimone} Let $A$ be an abelian subalgebra of a Lie algebra $L$. Suppose that $K = A + Fe_1$ is a subalgebra of $L$, and that there is an $x \in L$ such that $[x,K] \subseteq K$, but $[x,A] \not \subseteq A$. Then either $K$ is abelian or $K^2$ is one dimensional and $Z(K)$ has codimension at most one in $A$.
\end{propo}
\begin{proof} Let $e_2, \ldots, e_k$ be a basis for $A$ such that $e_1 = [x,e_2]$, say. Let $[x,e_j] = \sum_{i=1}^k \alpha_{ji} e_i$ for $1 \leq j \leq k$. Then
$[e_2,[x,e_j]] = \alpha_{j1}[e_2,e_1]$, so, for $2 \leq j \leq k$,
\[
0 = [x,[e_2,e_j]] = -[e_2,[e_j,x]] - [e_j,[x,e_2]] = \alpha_{j1}[e_2,e_1] - [e_j,e_1].
\]
Hence $[e_1,e_j] = \alpha_{j1}[e_1,e_2]$. It follows that $K^2 = F[e_1,e_2]$. Put $v_j = \alpha_{j1}e_2 - e_j$ for $3 \leq j \leq k$. Then $v_3, \dots, v_k \in Z(K) \cap A$.
\end{proof}
\bigskip

The above result deals with the case where an abelian subalgebra of maximal dimension has codimension one in an ideal of $L$.

\begin{coro}\label{c:codimone} Let $L$ be a supersolvable Lie algebra and let $A$ be an abelian subalgebra of maximal dimension in $L$. If $A \subset K$ where $K$ is an ideal of $L$ and $A$ has codimension $1$ in $K$, then $\alpha(L) = \beta(L)$.
\end{coro}
\begin{proof} If $A$ is an ideal of $L$ then the result is clear, so suppose that it is not an ideal of $L$. With the same notation as in Proposition \ref{p:codimone} the hypotheses of that result are satisfied. Then $v_3, \dots, v_k \in Z(K)$; in fact, the maximality of $A$ gives $Z(K) = Fv_3 + \dots + Fv_k$. Let $B/(Fv_3 + \dots + Fv_k)$ be a chief factor of $L$ with $B \subset K$. Then $B$ is an abelian ideal of $L$ with the same dimension as $A$. The result follows.
\end{proof}
\bigskip

Next we consider the situation where $L$ has a maximal subalgebra that is abelian: first when $L$ is any non-abelian Lie algebra and $F$ is algebraically closed, and then when $L$ is solvable but $F$ is arbitrary.

\begin{propo}\label{p:maxab} Let $L$ be a non-abelian Lie algebra of dimension $n$ over an algebraically closed field $F$ of any characteristic. Then $L$ has a maximal subalgebra $M$ that is abelian if and only if $L = A \dot{+} Ff$ for some $f \in gl(V)$, where $f \not \equiv 0$, $A$ is abelian and $[f,v] = -[v,f] = f(v)$. In particular, $\alpha(L) = \beta(L) = n-1$.
\end{propo}
\begin{proof} Suppose first that $L$ has a maximal subalgebra $M$ that is abelian. If $M$ is an ideal of $L$ we have finished. So suppose that $M$ is self-idealising, in which case $L^2$ is one-dimensional and $\phi(L) = 0$, by \cite[Proposition 3.2]{ind}. Write $L^2 = Fb$ and note that Asoc$L = L^2 \oplus Z(L)$. Now $L =$ Asoc$L \dot{+} C$, where $C$ is abelian, by \cite[Theorem 7.4]{frat}. For each $c \in C$ we have that $[c,b] = \lambda(c)b$, for some $\lambda(c) \in F$. Since $C \cap Z(L) = 0$, $C$ must be one-dimensional and the result follows.
\par
The converse is clear.
\end{proof}
\bigskip

The following is a generalisation of \cite[Proposition 3.1]{ind}

\begin{propo}\label{p:compsolv} Let $L$ be a solvable Lie algebra. Then $L$ has a maximal subalgebra $M$ that is abelian if and only if either
\begin{itemize}
\item[(i)] $L$ has an abelian ideal of codimension one in $L$; or
\item[(ii)] $L^{(2)} = \phi(L) = Z(L)$, $L^2/L^{(2)}$ is a chief factor of $L$, and $L$ splits over $L^2$.
\end{itemize}
\end{propo}
\begin{proof} By \cite[Proposition 3.1]{ind} it suffices to show that if $L$ has a maximal subalgebra $M$ that is not an ideal then $L^2$ is nilpotent. By maximality, $M$ is self-idealising, and from solvability there is a $k \geq 1$ such that $L^{(k)} \not \subseteq M$ but $L^{(k+1)} \subseteq M$. Then $L = M + L^{(k)}$, whence $L^2 \subseteq L^{(k)}$. It follows that $L^{(2)} \subseteq M$ and we have $M \subseteq C_L(L^{(2)})$. By maximality and the fact that $M$ is self-idealising, we get $C_L(L^{(2)}) = L$, which yields that $L^2$ is nilpotent. 
\end{proof}
\bigskip

Next we characterise solvable Lie algebras $L$ whose biggest abelian subalgebras have codimension two in $L$. The following proof relies on \cite[Propositions 3.1 and 5.1]{bc} which are only stated for Lie algebras over fields of characteristic zero. However, it is easy to see that this assunption is not used in their proofs, and that the results are, in fact, valid over an arbitrary field.

\begin{theor}\label{t:n-2} Let $L$ be a solvable Lie algebra of dimension $n$ with $\alpha(L) = n-2$, and let $A$ be an abelian subalgebra of dimension $n-2$. Then one of the following occurs:
\begin{itemize}
\item[(i)] $\beta(L) = n-2$;
\item[(ii)] $L = L^2 \dot{+} B$, where $B$ is an abelian subalgebra of $L$, $L^2$ is the three-dimensional Heisenberg algebra, $L^{(2)} = \phi(L) = Z(L)$ and $L^2/Z(L)$ is a two-dimensional chief factor of $L$ (in which case $\beta(L) \leq n-3$);
\item[(iii)] $A$ has codimension one in the nilradical, $N$, of $L$, which itself has codimension one in $L$. Moreover, $N^2$ is one dimensional, $Z(N)$ is an abelian ideal of maximal dimension and $\beta(L) = n-3$.
\end{itemize}
\end{theor}
\begin{proof} Let $A$ be a maximal abelian subalgebra of $L$ of dimension $n-2$ and suppose that (i) doesn't hold. 
\begin{description}
\item[(a)] Suppose first that $A$ is a maximal subalgebra of $L$. Then $L$ is as in Proposition \ref{p:compsolv}(ii) and $A$ is a Cartan subalgebra of $L$. Let $L = A \dot{+} L_1$ be the Fitting decomposition of $L$ relative to $A$. Then $L_1 \subseteq L^2$ and $\dim L_1 = 2$. Let $L_1 = Fx + Fy$. 
\par
If $[x,y] = 0$ then $L_1$ is an ideal of $L$ and $L/L_1$ is abelian, so $L^2 \subseteq L_1 \subseteq L^2$. This yields that $L$ is metabelian and $L^2$ is a two dimensional minimal ideal over which $L$ splits. It follows from Proposition \ref{p:metab} that $\beta(L) = n-2$, a contradiction.
\par
If $[x,y] \neq 0$, then $L^2 = F[x,y] + L_1$ and $F[x,y] \subseteq L^{(2)} = Z(L)$, so $F[x,y] = Z(L)$ and we have case (ii). Moreover, if $C$ is a maximal abelian ideal of $L$, then $Z(L) \subseteq C$ and $L^2 \not \subseteq C$. It follows that $C \cap L^2 = Z(L)$. If $\dim C = n-2$, then $\dim (L^2 + C) = \dim L^2 + \dim C - \dim C \cap L^2 = 3 + n-2 - 1 = n$, so $L = L^2 + C$. But then $L^2 = Z(L)$, a contradiction. Hence, $\beta(L) \leq n-3$.
\item[(b)] So suppose that $A$ is not a maximal subalgebra of $L$. Then $A \subset M \subset L$, where $\dim M = n-1$. Moreover, there is such a subalgebra $A$ of $M$ which is an ideal of $M$, by \cite[Proposition 3.1]{bc}. Suppose first that $A$ does not act nilpotently on $L$. Then the Fitting decomposition of $L$ relative to $A$ is $L = M \dot{+} L_1$, and $L_1$ is a one-dimensional ideal of $L$. Put $B = A \dot{+} L_1$, which is an ideal of $L$. Then $C_B(L_1)$ has codimension one in $B$ and so is an abelian ideal of codimension two in $L$. It follows that $\beta(L) = n-2$, a contradiction.
\par
Finally, suppose that $A$ is an ideal of $M$ and that $A$ acts nilpotently on $L$. Then there is a $k \geq 0$ such that $L($ad\,$A)^k \not \subseteq M$ but $L($ad\,$A)^{k+1} \subseteq M$. Let $x \in L($ad\,$A)^k \setminus M$, so $L = M \dot{+} Fx$. Suppose first that $M$ is not an ideal of $L$. Then the core of $M$, $M_L$ has codimension one in $M$, by \cite[Theorem 3.1 and 3.2]{amayo}. If $A = M_L$ then we have case (i), so suppose that $A \neq M_L$ and $M = A + M_L$. Then $[A,x] \subseteq M$ which implies that $[L,A] \subseteq M$ and $[L,M] = [L,M_L] + [L,A] \subseteq M$; that is, $M$ is an ideal of $L$.
\par
Let $N$ be the nilradical of $L$. If $N \subseteq A$ then $A \subseteq C_L(N) \subseteq N$, so $N = A$ and we have case (i) again. If $A \subset N$ then $N = L$ or we can assume that  $N = M$. If $A \not \subseteq N$ and $N \not \subseteq A$ then either $A + N = L$, in which case $L$ is nilpotent, or we can assume that $A + N = M$, in which case $M$ is a nilpotent ideal of $L$ and so $M = N$.
\par
If $L$ is nilpotent, then we have case (i), by \cite[Proposition 5.1]{bc}.If not, then we have case (iii) by Proposition \ref{p:codimone}.
\end{description}
\end{proof}

\begin{coro}\label{c:n-2supsolv} Let $L$ be a supersolvable Lie algebra of dimension $n$ with $\alpha(L) = n-2$. Then $\beta(L) = n-2$.
\end{coro}
\begin{proof} We use the same notation as in Theorem \ref{t:n-2} and show that cases (ii) and (iii) cannot occur. Clearly case (ii) cannot occur, since in that case $L^2/Z(L)$ is a two-dimensional minimal ideal of $L/Z(L)$. So suppose that case (iii) occurs.  Then $\dim Z(N) = n-3$. Since $L$ is supersolvable, there is an ideal $B \subset N$ of $L$ with $\dim(B/Z(N)) = 1$. But clearly $B$ is abelian, contradicting the maximality of $Z(N)$.
\end{proof}
\bigskip

Note that algebras of the type described in Theorem \ref{t:n-2} (ii) and (iii) do exist over the real field, as the following examples show.

\begin{ex}\label{e:one} Let $L$ be the four-dimensional Lie algebra over $\R$ with basis $e_1, e_2, e_3, e_4$ and non-zero products
\[
[e_1,e_2] = e_3, \hspace{.5cm} [e_1,e_3] = - e_2, \hspace{.5cm} [e_2,e_3] = e_4.
\]
Then this algebra is as described in Theorem \ref{t:n-2}(ii). For $L^2 = \R e_2 + \R e_3 + \R e_4$ is the three-dimensional Heisenberg algebra, $L = L^2 \oplus \R e_1$, $L^{(2)} = \R e_4 = Z(L) = \phi(L)$, and $L^2/Z(L)$ is a two-dimensional chief factor of $L$. This algebra has $\alpha(L) = 2$ and $\beta(L) = 1$. We could take $A = \R e_1 + \R e_4$, for example, but $\R e_4$ is the unique maximal abelian ideal of $L$.
\end{ex}

\begin{ex}\label{e:two} Let $L$ be the four-dimensional Lie algebra over $\R$ with basis $e_1, e_2, e_3, e_4$ and non-zero products
\[
[e_1,e_2] = e_2 - e_3, \hspace{.5cm} [e_1,e_4] = 2e_4, \hspace{.5cm} [e_1,e_3] = e_2 + e_3, \hspace{.5cm} [e_2,e_3] = e_4.
\]
Then this algebra is as described in Theorem \ref{t:n-2}(iii). For we could take $A = \R e_3 + \R e_4$, $N = \R e_2 + A$, so $N^2 = \R e_4$, $Z(N) = \R e_4$. We have $\alpha(L) = 2$ and $\beta(L) = 1$.
\end{ex}

\section{Nilpotent Lie algebras with $\alpha(L) = n-3$}
When $L$ is nilpotent of dimension $n$ and $\alpha(L) = n-3$ we obtain that $\alpha(L) = \beta(L)$, but only when $F$ has characteristic different from two.

\begin{theor}\label{t:n-3nilp} Let $L$ be a nilpotent Lie algebra of dimension $n$, over a field $F$ of characteristic different from two, with $\alpha(L) = n-3$. Then $\beta(L) = n-3$.
\end{theor}
\begin{proof} Let $A$ be an abelian subalgebra of $L$ with $\dim A = n-3$, let $N$ be a maximal subalgebra containing $A$ and suppose that $A$ is not an ideal of $L$. Then $N$ is an ideal of $L$, and $A$ is a maximal abelian subalgebra of $N$ of codimension $2$ in $N$. By \cite[Proposition 5.1]{bc} we can assume that $A$ is an ideal of $N$. Let $e_4, \ldots, e_n$ be any basis for $A$ and $L = Fe_1 + N$. We may suppose that $e_3 = [e_1,e_4] \notin A$; set $M = Fe_3 + A$. If $[e_1,M] \subseteq M$, then $M$ is an ideal of $L$, and the result follows from Corollary \ref{c:codimone}. Hence there exists $k$ such that $e_2 = [e_1,e_k] \notin M$ where $3 \leq k \leq n$ and $k \neq 4$. Clearly, $N = Fe_2 + M$. Let $[e_1,e_j] = \sum_{i=2}^n \alpha_{ji} e_i$ for $2 \leq j \leq n$. Then,
\begin{align}
[e_3,e_j] = [[e_1,e_4],e_j] & = - [[e_4,e_j],e_1] - [[e_j,e_1],e_4]  \\
   & = \alpha_{j2} [e_2,e_4] + \alpha_{j3} [e_3,e_4]  \hspace{.2cm} \hbox{for } j \geq 4.
\end{align}
Put $u_j = e_j - \alpha_{j3} e_4$  for $j \geq 5$. Then $[e_3,u_j] = \alpha_{j2} [e_2,e_4]$ for $j \geq 5$.
Therefore, we can choose the elements $e_5, \ldots, e_n$ in our basis so that 
\begin{align}
[e_3,e_j] & = \alpha_{j2} [e_2,e_4] \hbox{ and } \alpha_{j3} = 0 \hbox{ for } j \geq 5.
\end{align}
\medskip

{\bf Case 1} Suppose that $\alpha_{j2} = 0$ for all $j \geq 4$. So $[e_3,e_j] = 0$ for $j \geq 5$, $[L,A] \subseteq M$ and we can put $e_2 = [e_1,e_3]$. We have
\begin{align}
[e_2,e_j] = [[e_1,e_3],e_j] & = - [[e_3,e_j],e_1] - [[e_j,e_1],e_3]  \\
   & = \alpha_{j4} [e_4,e_3] \hspace{.2cm} \hbox{for } j \geq 5.
\end{align}
If $\alpha_{j4} = 0$ for all $j \geq 5$, then $\dim Z(N) \geq n-4$, and if we choose $B/Z(N)$ to be a chief factor of $L$ with $B \subset N$, $B$ is an abelian ideal of $L$ with $\dim B \geq n-3$.
\par
So suppose that $\alpha_{54} \neq 0$, say. Put $v_j = \alpha_{54}e_j - \alpha_{j4}e_5$ for $j \geq 6$.
 Then $[e_2,v_j] = 0$ for $j \geq 6$ and $\dim Z(N) \geq n-5$. So, we  can choose the terms $e_6, \ldots, e_n$ in the initial basis such that they belong to $Z(N)$. Let us note that $N^2$ is spanned by $[e_2,e_3]$, $[e_2,e_4]$ and $[e_3,e_4]$. Now
\begin{align}
[e_1,[e_2,e_5]] & = - [e_2,[e_5,e_1]] - [e_5,[e_1,e_2]] \nonumber \\
   & = \alpha_{54} [e_2,e_4] + \alpha_{55} [e_2,e_5] + \alpha_{22} [e_2,e_5], \hbox{ and } \nonumber \\
[e_1,[e_4,e_3]] & = - [e_4,[e_3,e_1]] - [e_3,[e_1,e_4]] \nonumber \\
    & = [e_4,e_2]. \nonumber
    \end{align}
It follows from $(5)$ that $$2 \alpha_{54} [e_2,e_4] = (\alpha_{55} + \alpha_{22}) \alpha_{54} [e_3,e_4],$$
so $\dim N^2 \leq 2$. Let $B/Z(N)$ be a chief factor of $L$ with $B \subset N$. Then $B$ is an abelian ideal of $L$ of dimension at least $n-4$. Suppose that $\dim B = n-4$ and that this is a maximal abelian ideal of $L$. Then $C_L(B) = B$. Put $B = Z(N) + Fb$. Now $[e_3,b]$, $[e_4,b]$, $[e_5,b] \in N^2$.
 Since $\dim N^2 \leq 2$ these elements are linearly dependent. Hence we have $\beta_3 [e_3,b] + \beta_4 [e_4,b] + \beta_5 [e_5,b] = 0$ for some $\beta_3, \beta_4, \beta_5 \in F$, not all zero. It follows that $\beta_3 e_3 + \beta_4 e_4 + \beta_5 e_5 \in C_L(B) = B$, whence $B = F(\beta_3 e_3 + \beta_4 e_4 + \beta_5 e_5) + Z(N)$. Now $e_5 \in C_L(B) = B$, so $B = Z(N) + Fe_5$. But then $e_4 \in C_L(B) \setminus B$, a contradiction. It follows that there is an abelian ideal of dimension $n-3$.
\medskip

{\bf Case 2} Suppose that $\alpha_{52} \neq 0$, say, so $[e_1,e_5] \notin M$. Put $e_2 = [e_1,e_5]$, so that $\alpha_{52} = 1$, $\alpha_{5j} = 0$ for $j \neq 2$; put also $v_j = e_j - \alpha_{j2}e_5$ for $j \geq 6$. Then $[e_3,v_j] = 0$ for $j \geq 6$. We also have
\begin{align}
[e_1,v_j] & =  [e_1,e_j] - \alpha_{j2}[e_1,e_5] \nonumber \\
   & \in A  \hbox{ for } j \geq 6, \hspace{2cm} \hbox{using } (3)  \nonumber
\end{align}
so $[e_2,v_j] = [[e_1,e_5],v_j] = - [[e_5,v_j],e_1] - [[v_j,e_1],e_5] = 0$ for $j \geq 6$. So again $\dim Z(N) \geq n-5$, we can choose $e_6, \dots, e_n$ in our original basis to belong to $Z(N)$, and $N^2$ is spanned by $[e_2,e_3]$, $[e_2,e_4]$, $[e_2,e_5]$ and $[e_3,e_4]$. Now
\begin{align}
[e_1,[e_2,e_4]] & = -[e_2,[e_4,e_1]] - [e_4,[e_1,e_2]] \nonumber \\
   & = [e_2,e_3] + \alpha_{22} [e_2,e_4] + \alpha_{23} [e_3,e_4],
\end{align}
and
\begin{align}
[e_1,[e_3,e_5]] & = - [e_3,[e_5,e_1]] - [e_5,[e_1,e_3]] \nonumber \\
     & = [e_3,e_2] + \alpha_{32}[e_2,e_5] + \alpha_{33} [e_3,e_5].
\end{align}
Since $[e_3,e_5] = [e_2,e_4]$ this yields
\begin{align}
2[e_2,e_3] = (\alpha_{33} - \alpha_{22}) [e_2,e_4] - \alpha_{23} [e_3,e_4] + \alpha_{32}[e_2,e_5],
\end{align}
so $N^2$ is spanned by $[e_2,e_4]$, $[e_2,e_5]$ and $[e_3,e_4]$.
\par

We have $N^2 \subseteq A$, since $\dim N/A = 2$. Suppose first that $N^2 \not \subseteq Z(N)$. Then choose $B \subseteq N^2 + Z(N)$ such that $B/Z(N)$ is a chief factor of $L$. Then $B \subset A$ and $A \subseteq C_L(B) \setminus B$. It follows that $L$ has an abelian ideal of dimension $n-3$.
\par

So consider now the case where $N^3 = 0$. We have $\dim L/(L^2 + Z(N)) \leq 3$ since $e_2, e_3 \in L^2$. Suppose first that $\dim L/(L^2 + Z(N)) = 3$, so that $L^2 + Z(N) = Fe_2 + Fe_3 + Z(N) = D$, say. This is an ideal of $L$, so $\alpha_{24} = \alpha_{25} = \alpha_{34} = \alpha_{35} = 0$. Now
\begin{align}
[e_1,[e_2,e_3]] & = - [e_2,[e_3,e_1]] - [e_3,[e_1,e_2]] \nonumber \\
   & = \alpha_{33} [e_2,e_3] + \alpha_{22} [e_2,e_3] \nonumber \\
   & = (\alpha_{33} + \alpha_{22}) [e_2,e_3].
\end{align}
Since $L$ is nilpotent we must have $\alpha_{33} = - \alpha_{22}$. Now $$\dim ([L,D]+Z(N))/Z(N) \leq 1,$$ so $[e_1,e_3] + Z(N) = \lambda [e_1,e_2] + Z(N)$ for some $\lambda \in F$. It follows that $\alpha_{32} = \lambda \alpha_{22}$, $- \alpha_{22} = \alpha_{33} = \lambda \alpha_{23}$ and $B = F(e_3 - \lambda e_2) + Z(N)$ is an abelian ideal of $L$. Suppose first that $\lambda \neq 0$. Then
$$\begin{array}{l}
[e_3 - \lambda e_2, \lambda e_2 + e_3 - \lambda \alpha_{23} e_4 - \alpha_{32} e_5]   \\
=  - 2 \lambda[e_2,e_3] + (\lambda^2 \alpha_{23}-\alpha_{32})[e_2,e_4]-\lambda \alpha_{23} [e_3,e_4] + \lambda \alpha_{32} [e_2,e_5]   \\
=  \lambda(-2[e_2,e_3] + (\alpha_{33} - \alpha_{22}) [e_2,e_4] - \alpha_{23} [e_3,e_4] + \alpha_{32}[e_2,e_5])  \\
=  0,  
\end{array}$$
using $(8)$. It follows that $C_L(B) \neq B$, so $B$ is not a maximal abelian ideal of $L$ and the result holds.
\par
If $\lambda = 0$ then $\alpha_{32} = \alpha_{33} = \alpha_{22} = 0$ and $(8)$ becomes $2[e_2,e_3] = - \alpha_{23} [e_3,e_4]$. But this implies that $[2e_2 - \alpha_{23}e_4,e_3] = 0$, and $C_L(B) \neq B$ again.
\par

So suppose now that $\dim L/(L^2 + Z(N)) = 2$. Then there is an $n_1 \in N$ such that $L = Fe_1 + Fn_1 + Fn_2 + Fn_3 + Fn_4 + Z(N)$ where $n_2 = [e_1,n_1]$, $n_3 = [e_1,n_2]$, $n_4 = [e_1,n_3]$, $[e_1,n_4] \in Z(N)$. Now
\begin{align}
[e_1,[n_1,n_2]] & = - [n_1,[n_2,e_1]] - [n_2,[e_1,n_1]] = [n_1,n_3] \nonumber \\
[e_1,[n_1,n_3]] & = - [n_1,[n_3,e_1]] - [n_3,[e_1,n_1]] = [n_1,n_4] + [n_2,n_3] \nonumber \\
[e_1,[n_1,n_4]] & = - [n_1,[n_4,e_1]] - [n_4,[e_1,n_1]] = [n_2,n_4] \nonumber \\
[e_1,[n_2,n_4]] & = - [n_2,[n_4,e_1]] - [n_4,[e_1,n_2]] = [n_3,n_4] \nonumber \\
[e_1,[n_2,n_3]] & = - [n_2,[n_3,e_1]] - [n_3,[e_1,n_2]] = [n_2,n_4]. \nonumber
\end{align}
Since $\dim N^2 \leq 3$ we have
$$ 0 = [e_1,[e_1,[e_1,[n_1,n_2]]]] = 2[n_2,n_4].$$ Clearly $B = Fn_4 + Z(N)$ is an abelian ideal, and $n_2 \in C_L(B) \setminus B$, which completes the proof.
\end{proof}

\bigskip

The restriction on the characteristic in the above result is necessary, as the following example shows.

\begin{ex}\label{e:three} Let $L$ be the nine-dimensional Lie algebra, over any field $F$ of characteristic two, with basis $e_1, e_2, e_3, e_4, e_5, e_6, e_7, e_8, e_9$ and non-zero products
\[
[e_1,e_2] = e_6, \hspace{.5cm} [e_1,e_3] = e_2, \hspace{.5cm} [e_1,e_4] = e_3, \hspace{.5cm} [e_1,e_5] = e_4, \hspace{.5cm} [e_1,e_8] = e_7,
\]
\[
[e_1,e_9] = e_8, \hspace{.5cm} [e_2,e_3] = e_7, \hspace{.5cm} [e_2,e_4] = e_8, \hspace{.5cm} [e_2,e_5] = e_9, \hspace{.5cm} [e_3,e_4] = e_9.
\]
This is a nilpotent Lie algebra whose abelian subalgebras of maximal dimension are
$$F(e_3 + \lambda e_4) + Fe_5 + Fe_6 + Fe_7 + Fe_8 + Fe_9 \hspace{.5cm} \hbox{and}$$
$$F(\lambda e_3 + e_4) + Fe_5 + Fe_6 + Fe_7 + Fe_8 + Fe_9 \hspace{.5cm} (\lambda \in F),$$
so $\alpha(L) = 6$. However, none of these are ideals of $L$; in fact the abelian ideal of maximal dimension is
$$Fe_2 + Fe_6 + Fe_7 + Fe_8 + Fe_9,$$
so $\beta(L) = 5$.
\end{ex}

Note that the above example is valid over an algebraically closed field of characteristic two, and so shows that \cite[Proposition 2.6]{bc} does not hold over such fields, even when $L$ is nilpotent.
\par

The results above prompt the following questions. 
\begin{itemize}
\item[1.] Does Theorem \ref{t:n-3nilp} hold for supersolvable Lie algebras $L$?
\item[2.] Let $L$ be a supersolvable/nilpotent Lie algebra with $\alpha(L) = n-k$ containing an abelian subalgebra $A$ of maximal dimension, and let $N$ be a maximal subalgebra containing $A$ that is an ideal of $L$.
\begin{itemize} 
\item[(i)] Is it true that $\dim Z(N) \geq n - 2k + 1$?
\item[(ii)] Is it true that $\dim N^2 \leq k - 1$?
\item[(iii)] If they are true, do 1 and 2 imply that $\beta(L) = n - k$?
\end{itemize}
\end{itemize}
However, we have seen that restrictions on the underlying field are necessary for any of these to be true.

\end{document}